\documentclass[11pt]{amsart}

\usepackage{amsmath, amssymb}
\usepackage{xypic}
\usepackage{hyperref}

\newcounter{my_enumerate_counter}
\newcommand{\pushcounter}{\setcounter{my_enumerate_counter}{\value{enumi}}}
\newcommand{\popcounter}{\setcounter{enumi}{\value{my_enumerate_counter}}}

\newcommand{\bbZ}{\mathbb Z}

\DeclareMathOperator{\Triv}{Triv}
\DeclareMathOperator{\half}{half}
\DeclareMathOperator{\pos}{pos}
\DeclareMathOperator{\nor}{nor}
\newcommand{\bS}{\mathbf S} 
\newcommand{\bfx}{\mathbf x} 
\newcommand{\bfH}{\mathbf H} 
\DeclareMathOperator{\CRx}{CR_\bfx}

\DeclareMathOperator{\Sigx}{\Sigma_\bfx}

\DeclareMathOperator{\Fin}{Fin}
\DeclareMathOperator{\intG}{int_G} 
\newcommand{\intGx}[1]{\mathrm{int}_{G\rs #1}} 
\newcommand{\bDelta}{\mathbf \Delta}
\newcommand{\bSigma}{\mathbf \Sigma}
\newcommand{\bPi}{\mathbf \Pi}
\DeclareMathOperator{\ZFC}{ZFC^*}

\newcommand{\Qx}{\bbQ_{\bfx}}
\newcommand{\forces}{\Vdash}

\newcommand{\bb}{\mathbf b}

\newcommand{\bfC}{\mathbf C}

\newcommand{\bbS}{{\mathbb S}}

\newcommand{\bfS}{\mathbf S}

\newcommand{\bbN}{{\mathbb N}}

\newcommand{\bbQ}{\mathbb Q}
\newcommand{\bbR}{\mathbb R}
\newcommand{\cJ}{{\mathcal J}}
\newcommand{\cI}{{\mathcal I}}
\newcommand{\calR}{{\mathcal R}}
\newcommand{\cW}{{\mathcal W}}

\newcommand{\cX}{{\mathcal X}}
\newcommand{\calH}{{\mathcal H}}

\newcommand{\cZ}{{\mathcal Z}}

\newcommand{\fc}{\mathfrak c} 
\newcommand{\fd}{\mathfrak d} 
\newcommand{\fe}{\mathfrak e} 
\newcommand{\rs}{\restriction}

\newcommand{\cF}{\mathcal F}

\newcommand{\cB}{\mathcal B}
\newcommand{\cK}{\mathcal K}

\newcommand{\bbP}{\mathbb P}

\newtheorem{thm}{Theorem}[section]
\newtheorem{theorem}{Theorem}
\newtheorem{coro}[thm]{Corollary}
\newtheorem{corollary}[theorem]{Corollary}

\newtheorem{question}[thm]{Question}

\newtheorem{lemma}[thm]{Lemma}

\newtheorem{prop}[thm]{Proposition}

\theoremstyle{definition}

\newtheorem{definition}[thm]{Definition}

\newtheorem{problem}[thm]{Problem}

\DeclareMathOperator{\cf}{cf} 

\newcommand{\cP}{\mathcal P} 
\renewcommand{\vec}{\bar}

\DeclareMathOperator{\supp}{supp}

\newcommand{\iteration}{(\bbP_\xi, \dot\bbQ_\eta: \xi\leq \kappa, \eta<\kappa)}
\title{Trivial automorphisms}

\author{Ilijas Farah}

\address{Department of Mathematics and Statistics\\
York University\\
4700 Keele Street\\
North York, Ontario\\ Canada, M3J
1P3\\
and Matematicki Institut, Kneza Mihaila 35, Belgrade, Serbia}

\urladdr{http://www.math.yorku.ca/$\sim$ifarah}

\email{ifarah@mathstat.yorku.ca}

\author{Saharon Shelah}

\address{The Hebrew University of Jerusalem\\
Einstein Institute of Mathematics\\
Edmond J. Safra Campus, Givat Ram\\
Jerusalem 91904, Israel\\ and \\
Department of Mathematics\\
Hill Center-Busch Campus\\
Rutgers, The State University of New Jersey\\
110 Frelinghuysen Road\\
Piscataway, NJ 08854-8019 USA}

\email{shelah@math.huji.ac.il}
\urladdr{http://shelah.logic.at/}

\thanks{The first author was  partially supported by NSERC.}

\thanks{The second author would like to thank the Israel Science Foundation for
  partial support of this research (Grant no. 710/07).
No. 987 on 
Shelah's list of publications.}

\date{\today}

\begin{document}

\begin{abstract} We prove that the statement `For all  Borel ideals  $\cI$ and $\cJ$ on $\omega$, 
every isomorphism between Boolean algebras $\cP(\omega)/\cI$ and $\cP(\omega)/\cJ$ has a continuous representation' is relatively consistent with ZFC. In this model  
every isomorphism 
between $\cP(\omega)/\cI$ and any other quotient $\cP(\omega)/\cJ$ over a Borel ideal is trivial
for a number of Borel ideals   
$\cI$ on~$\omega$. 

We can also assure that  the dominating number,  $\fd$, is equal to $\aleph_1$ and that $2^{\aleph_1}>2^{\aleph_0}$.  
Therefore the  Calkin algebra has outer automorphisms while all automorphisms of $\cP(\omega)/\Fin$ 
are trivial. 

Proofs rely on delicate analysis of names for reals in a 
countable support iteration of suslin proper forcings. 
\end{abstract} 

\maketitle
\tableofcontents
\section{Introduction}
We start with a fairly general setting. 
Assume $X/I$ and $Y/J$ are quotient structures (such as groups, Boolean algebras, C*-algebras,\dots) with $\pi_I$ and $\pi_J$ denoting the respective quotient maps. 
Also assume  $\Phi$ is an isomorphism between 
$X/I$ and $Y/J$. A \emph{representation} of $\Phi$ is a map $F\colon X\to Y$ such that the diagram
\[
\diagram
X \rto^{\Phi_*} \dto_{\pi_I} & Y\dto_{\pi_J} \\
X/I\rto_{\Phi} & Y/J
\enddiagram
\]
commutes. Since representation is not required to have any algebraic properties its 
existence follows from the Axiom of Choice and is therefore inconsequential to the 
relation of $X/I$, $X/J$  and $\Phi$. 

We shall say that $\Phi$ is \emph{trivial} if it has a representation that is itself
a homomorphism between $X$ and $Y$. Requiring a representation to be
an isomorphism itself would be too strong since in many situations of interest there
exists an isomorphism which has a representation that is a homomorphism but does 
not have one which is an isomorphism. 

In a number of cases of interest  
 $X$ and $Y$ are structures of cardinality of the continuum and 
quotients $X/I$ and $Y/J$ are countably saturated in the model-theoretic sense (see e.g., \cite{ChaKe}). 
In this situation Continuum Hypothesis, CH,  makes it possible to  
use a diagonalization to construct
nontrivial automorphisms  of $X/I$ and, if the quotients are elementarily equivalent, 
an isomorphism between $X/I$ and  $Y/J$. 
For example, CH implies that 
Boolean algebra $\cP(\omega)/\Fin$ has nontrivial automorphisms
(\cite{Ru}) and Calkin algebra has outer automorphisms (\cite{PhWe:Calkin} or \cite[\S 1]{Fa:All}). 
This is by no means automatic and for example the quotient group $S_\infty/G$ (where $G$ is the
subgroup consisting of finitely supported permutations) has 
the group of outer automorphisms isomorphic to $\bbZ$ and all of its automorphisms are trivial 
(\cite{AlCoMac}). 
 Also, some quotient Boolean algebras of the form $\cP(\omega)/\cI$ 
for Borel ideals $\cI$ are not countably saturated 
and it is unclear whether  nontrivial automorphisms  exist (see \cite{Fa:CH}).
A construction of isomorphism between quotients over two different density ideals
that are not countably saturated in classical sense in \cite{JustKr} should be revisited using the 
logic of metric structures developed in~\cite{BYBHU}. 
As observed in \cite{JustKr}, these two quotients have the natural 
structure of complete metric spaces and when considered as models of the logic of metric structures
two algebras are countably saturated. This fact can be extracted from the proof in 
\cite{JustKr} or from its generalization given in \cite{Fa:CH}. 

We shall consider the opposite situation, but only after noting that by Woodin's $\Sigma^2_1$ 
absoluteness theorem (\cite{Wo:Beyond}, \cite{Lar:Stationary}) 
Continuum Hypothesis provides the optimal context for finding nontrivial 
isomorphisms whenever $X$ and  $Y$ have Polish space structure with Borel-measurable 
operations and 
 $I$ and $J$ are Borel ideals (see \cite[\S 2.1]{Fa:AQ}).

The line of research to which the present paper belongs was started by the 
second author's proof  that the assertion `all automorphisms of $\cP(\omega)/\Fin$ are trivial' 
is relatively consistent with ZFC (\cite{Sh:b}). A weak form of this conclusion
 was extended to some other Boolean algebras of the 
form $\cP(\omega)/\cI$  in \cite{Just:Modification} and~\cite{Just:Repercussions}. 
 This line of research took a new turn when it was realized that 
 forcing axioms imply all isomorphisms  between quotients over Boolean 
 algebras $\cP(\omega)/\cI$, for certain Borel ideals $\cI$, 
 are trivial (\cite{ShSte:PFA}, \cite{Ve:OCA}, \cite{Just:WAT},  
 \cite{Fa:AQ}, \cite{Fa:Luzin}). 
 The first author conjectured in \cite{Fa:Rigidity} that the Proper Forcing Axiom, PFA, 
 implies all isomorphisms between any two quotient algebras of the form $\cP(\omega)/\cI$, 
 for a Borel ideal $\cI$, are trivial. 
 This conjecture naturally splits in following two rigidity conjectures. 
 \begin{enumerate}
 \item [(RC1)] PFA implies every isomorphism has a continuous representation, and
 \item [(RC2)] Every isomorphism with a continuous representation is trivial. 
 \end{enumerate}
 Noting that in our situation Shoenfield's Absoluteness Theorem implies that~(RC2) cannot be changed
 by forcing and that no progress on it has been made in the last ten years, we shall 
 concentrate on (RC1). 

In the present paper we construct a forcing extension in which all isomorphisms between 
Borel quotients have continuous representations. This does not confirm (RC1) but it does
give some positive evidence towards it. 

The assumption of the existence of a measurable cardinal in the following result is used only to assure sufficient forcing-absoluteness\footnote{More precisely, we need to know that in all forcing extensions 
by a small proper forcing all $\bSigma^1_2$ sets have the property of Baire, 
 $\bPi^1_2$-unformization and that all $\bPi^1_2$ sets have the Property of Baire. By Martin--Solovay (\cite{MaSo:Basis}) it suffices to assume that $H(\mathfrak c^+)^\#$ exists}  and it is
very likely unnecessary. 

\begin{theorem}\label{T0} Assume there exists a measurable cardinal. Then there is a forcing
extension in which all of the following are 
true. 
\begin{enumerate}
\item \label{T0.1} Every automorphism of a  quotient Boolean algebra  
$\cP(\omega)/\cI$ over a Borel ideal $\cI$ has a continuous representation. 
\item \label{T0.2} 
Every isomorphism between quotient Boolean algebras $\cP(\omega)/\cI$ and $\cP(\omega)/\cJ$
over Borel ideals has a continuous representation. 
\item Every homomorphism between  quotient Boolean algebras
$\cP(\omega)/\cI$ over Borel ideals  has a locally continuous representation. 
\item \label{T0.4} The dominating number, $\fd$, is equal to $\aleph_1$.
\item \label{T0.wCH} 
All of the above, and  in addition we can have either $2^{\aleph_0}=2^{\aleph_1}$ 
or~$2^{\aleph_0}<2^{\aleph_1}$. 
\end{enumerate}
\end{theorem} 

Proof of Theorem~\ref{T0} will occupy most of the present paper (see
\S\ref{S.Plan} and \S\ref{S.ProofT1} for an outline). 
By the above the consistency of the conclusion of the 
full rigidity conjecture, `it  is relatively consistent with ZFC that all isomorphisms between 
quotients over Borel ideals are trivial' reduces to (RC2) above.

\begin{corollary} \label{C.Calk} It is relatively consistent with ZFC + `there exists a measurable cardinal'  that 
all automorphisms of $\cP(\omega)/\Fin$ are trivial while the Calkin algebra has  outer  automorphisms. 
In addition, the  corona  of every separable, stable C*-algebra has outer automorphisms. 
\end{corollary} 

\begin{proof} By the above, the triviality of all automorphisms
of $\cP(\omega)/\Fin$  together with  $\fd=\aleph_1$ and
Luzin's weak Continuum Hypothesis,  
is relatively consistent with ZFC + `there exists a measurable cardinal'. By \cite[\S 1]{Fa:AQ}, the two latter assumptions imply the 
existence of an outer automorphism of the Calkin algebra. 
An analogous  result for coronas of some other C*-algebras, including separable stable algebras, 
is proved in   \cite{CoFa:Automorphisms}.  
\end{proof}

If $\alpha$ is an indecomposable countable ordinal, the \emph{ordinal ideal} $\cI_{\alpha}$ is
the ideal on $\alpha$ consisting of all subsets of $\alpha$ of strictly smaller order type. 
If $\alpha$ is multiplicatively indecomposable, then the \emph{Weiss ideal} $\cW_\alpha$ 
is the ideal of all subsets of $\alpha$ that don't include a closed copy of $\alpha$ in 
the ordinal topology. 
See \cite{Fa:AQ} for more on these ideals and the definition of nonpatholigical analytic p-ideals.

\begin{corollary} It is relatively consistent with  
 ZFC + `there exists a measurable cardinal' that 
every isomorphism between $\cP(\omega)/\cI$ and $\cP(\omega)/\cJ$ is trivial whenever
$\cI$ is Borel and $\cJ$  
 is in any of the following classes of ideals 
is trivial: 
\begin{enumerate}
\item Nonpathological analytic p-ideals, 
\item Ordinal ideals, 
\item Weiss ideals, 
\end{enumerate}
In particular, quotient over an ideal of this sort and any other Borel ideal 
can be isomorphic if and only if the ideals are isomorphic.  
\end{corollary} 

\begin{proof} 
If an isomorphism $\Phi\colon \cP(\omega)/\cI\to \cP(\omega)/\cJ$ 
has a continuous representation and $\cJ$ is in one of the above classes, then 
$\Phi$ is trivial. This was proved in \cite{Fa:AQ}, \cite{KanRe:New}  and \cite{KanRe:Ulam}. 
\end{proof}

In the presence of sufficient large cardinals and
forcing absoluteness, the forcing notion used in the proof 
Theorem~\ref{T0} gives a stronger consistency result. 
Universally Baire sets of reals were defined in \cite{FeMaWo:uB}
and well-studied since. A reader not familiar with the theory of universally Baire sets 
 may safely skip all references
to them. 

\begin{theorem} \label{T0.uB} Assume there are class many Woodin cardinals. 
Then all conclusions of Theorem~\ref{T0}  
hold simultaneously for arbitrary universally Baire ideals in 
place of Borel ideals. 
\end{theorem} 

The proof of Theorem~\ref{T0.uB} will be  sketched in \S\ref{S.uB}. 

\subsection{The plan}\label{S.Plan} We now roughly outline the proof of Theorem~\ref{T0}. 
Starting from a model of CH force with a countable support iteration of
creature forcings $\Qx$ (\S\ref{S.CR}) and standard posets for adding a Cohen real, $\calR$. 
The iteration has length $\aleph_2$ and each of these forcings occurs on a stationary set of 
ordinals of uncountable cofinality specified in the ground model. 
Forcing $\Qx$ adds a real which destroys homomorphisms between Borel quotients that are not 
locally topologically trivial (\S\ref{S.Defs}).  

Now consider an isomorphism $\Phi\colon \cP(\bbN)/\cI\to \cP(\bbN)/\cJ$ between quotients
over Borel ideals $\cI$ and $\cJ$. 
By the standard reflection arguments (\S\ref{S.Ref}) and using the above property of $\Qx$ 
we show that $\Phi$ is locally topologically trivial. 
Finally, a locally topologically trivial automorphism that survives adding random reals has a continuous 
representation (Lemma~\ref{L3}). 

In order to make all this work we need to assure that the forcing iteration is sufficiently 
definable. In particular we have the continuous reading of names  (\S\ref{S.CRN}). 
A simplified version of the forcing notion with additional applications appears in 
\cite{Gha:FDD}. 

The forcing notion used to prove   Theorem~\ref{T0.uB} is identical to the one used in 
Theorem~\ref{T0}. 
With an  additional absoluteness assumptions the main result of this paper can be extended
to a class of ideals larger than Borel. 
We shall need the fact that, assuming the existence of 
class many Woodin cardinals, all projective sets of reals are universally Baire
and, more generally, that every set projective in a universally Baire set is universally Baire. 
Proofs of these results use Woodin's stationary tower forcing and they can be found in 
\cite{Lar:Stationary}.

\subsection{Notation and conventions}
Following \cite{Sh:630} we denote the theory obtained from ZFC by removing the power set axiom 
and adding `$\beth_\omega$ exists' by 
$\ZFC$. 

We frequently simplify and abuse the notation and write 
$\Phi\rs a$ instead of the correct $\Phi\rs \cP(a)/(a\cap \cI)$ when 
 $\Phi\colon \cP(\omega)/\cI\to \cP(\omega)/\cJ$ 
 and $a\subseteq \omega$. 

If $\cI$ is an ideal on $\bbN$ then $=^\cI$ denotes the equality modulo $\cI$ on $\cP(\bbN)$.  

As customary in set theory, interpretation of the
 symbol $\bbR$ (`the reals') depends from the context. 
It  may denote $\cP(\omega)$, $\omega^\omega$, or any other recursively presented Polish space. 
Set-theoretic terminology and notation are standard, as in \cite{Ku:Book}, \cite{Sh:f} or \cite{Kana:Book}. 

\section{Local triviality} \label{S.Defs}
\label{S.LBP} 

We start by   gathering a couple of soft results about representations of homomorphisms. 
A homomorphism $\Phi\colon \cP(\omega)/\cI\to \cP(\omega)/\cJ$ 
is $\bDelta^1_2$ if the set 
\[
\{(a,b): \Phi([a]_\cI)=[b]_\cJ\}
\]
 includes a  $\bDelta^1_2$ set $\cX$ such that  for every $a$ there exists $b$ for which $(a,b)\in \cX$.  
  We similarly define when $\Phi$ is  
Borel, $\bPi^1_2$, or in any other pointclass.

For a  homomorphism 
$\Phi\colon \cP(\omega)/\cI\to \cP(\omega)/\cJ$ 
consider the ideals
\[
\Triv_\Phi^0=\{a\subseteq\omega: \Phi\rs a\text{ is trivial}\},
\]
\[
\Triv_\Phi^1=\{a\subseteq\omega: \Phi\rs a\text{ has a continuous representation}\}, 
\]
and
\[
\Triv_\Phi^2=\{a\subseteq\omega: \Phi\rs a\text{ is $\bDelta^1_2$}\}. 
\]
We say that $\Phi$ is \emph{locally trivial} if $\Triv_\Phi^0$ is nonmeager,  
that it is \emph{locally topologically trivial} if $\Triv_\Phi^1$ is nonmeager and
hat it is \emph{locally $\bDelta^1_2$} if $\Triv_\Phi^2$ is nonmeager.

 By \cite[Theorem~3.3.5]{Fa:AQ} a fairly weak consequence of PFA 
  implies every homomorphism 
between quotients over Borel p-ideals is locally continuous (and a bit more). 
See \cite{Fa:AQ} for additional definitions.

 By a well-known result of Jalali--Naini and Talagrand (for a proof 
see \cite{BarJu:Book} or \cite[Theorem~3.10.1]{Fa:AQ}) for each meager 
ideal $\cI$ that includes $\Fin$
there is a partition $\vec I=(I_n: n\in \omega)$ of $\omega$ into finite intervals 
such that for every infinite $c\subseteq \omega$ the set 
$\vec I_c=\bigsqcup_{n\in C} I_n$ is positive.
In other words, the ideal $\cI$ is meager if and only if for some partition $\vec I$ of $\omega$ into finite intervals 
$\cI$ is included in the hereditary $F_\sigma$ set  
\[
\calH(\vec I)=\{a\subseteq \omega: (\forall^\infty n) I_n\not \subseteq a\}.
\]
We say that $\vec I$ \emph{witnesses} $\cI$ is meager. 
If $\cI$ is an ideal that has the property of Baire and includes $\Fin$, then 
it is necessarily meager.

 The following is a well-known consequence of the above. 
 
 \begin{lemma} \label{L.y1} 
 Assume $\cI$ is a Borel ideal and $\cK$ is a nonmeager ideal. 
 Then for every $c\in \cI^+$ there is $d\in \cK$ such that $c\cap d\in \cI^+$. 
 \end{lemma} 

\begin{proof} Since the ideal $\cI\cap \cP(c)$ is a proper Borel ideal on $c$, 
it is meager and we can find a partition of $c$ into intervals $c=\bigsqcup_n I_n$
such that $\bigsqcup_{n\in y} I_n\notin \cI$ for every infinite $y\subseteq \omega$. 
Let $\omega=\bigsqcup_n J_n$ be a partition such that $J_n\cap c=I_n$ for all $n$. 
Since $\cK$ is nonmeager, there is an infinite $y$ such that $d=\bigcup_{n\in y} J_n$
belongs to $\cK$. Then $d\cap c=\bigsqcup_{n\in y} I_n$ is not in $\cI$ and
therefore $d$ is as required. 
\end{proof}

The assumption of the following lemma follows from the assumption that there exists 
a measurable cardinal by \cite{MaSo:Basis}. 

\begin{lemma}\label{L-1} 
Assume  that 
all $\bSigma^1_2$ sets of reals have the property of Baire.
If a homomorphism 
$\Phi\colon \cP(\omega)/\cI\to \cP(\omega)/\cJ$ 
is $\bDelta^1_2$ then it has a continuous representation. \end{lemma} 

\begin{proof} By the Novikov--Kondo--Addison uniformization theorem, 
$\Phi$ has a $\bSigma^1_2$ representation. Since this map is Baire-measurable, 
by 
 a well-known fact 
(e.g.,  \cite[Lemma~1.3.2]{Fa:AQ}) 
$\Phi$ has a continuous representation. 
\end{proof}

Given a partition $I=(I_n: n\in \omega)$ of $\omega$ into finite intervals we say that a forcing
notion $\bbP$ \emph{captures} $I$ if there is a $\bbP$-name $\dot r$ for a subset of $\omega$
such that for every $p\in \bbP$ there is an infinite $c\subseteq \omega$ with the following 
property: 
\begin{enumerate} 
\item For every $d\subseteq \bigcup_{n\in c} I_n$ there is $q_d\leq p$ such that 
$q_d$ forces 
\[
\textstyle \dot r\cap \bigcup_{n\in c} \check I_n=\check d.
\] 
\end{enumerate}
 By $[a]_\cI$ we denote the 
equivalence class of set $a$ modulo the ideal $\cI$. When the ideal is clear from the context 
we may write $[a]$ instead of $[a]_\cI$.

\section{Creatures}\label{S.CR}  

Two suslin proper forcing notions are used in theÊ proof of Theorem~\ref{T0}. 
One is the Lebesgue measure 
algebra, $\calR$. 
The other shall be described in the present section 
It is a creature forcing (for background see \cite{RoSh:470}).

Fix a partition $I=(I_n: n\in \omega)$ of $\omega$ into consecutive finite intervals. 
Also fix another fast partition $J=(J_n: n\in \omega)$ into consecutive finite intervals. 
For $s\subseteq \omega$ write 
\[
\textstyle I_s=\bigcup_{j\in s} I_j\text{ and } 
I_{<n}=\bigcup_{j<n} I_j. 
\]
 Let $\bfx$ denote the pair $(I,J)$, 
called `relevant parameter.' 
Define $(\CRx,\Sigx)$ as follows (in terms of \cite{RoSh:470}, this will be a `creating pair').  

Let $\fc\in\CRx$ if 
\[
\fc=(n_\fc,u_\fc,\eta_\fc,\cF_\fc,m_\fc,k_\fc)
\]
(we omit the subscript $\fc$ whenever it is clear from the context) 
provided the following conditions hold
\begin{enumerate}
\item $u\subseteq J_n$, 
\item $\eta: I_u\to \{0,1\}$, 
\item $\cF\subseteq \{0,1\}^{I_{J_n}}$ and each $\mu\in \cF$ extends $\eta$, 
\item\label{I.CR.k}  $k\leq |J_n|-|u|$, 
\item if $v\subseteq J_n\setminus u$, $|v|\leq k$, and 
$\nu\colon I_v\to \{0,1\}$ then some $\mu\in \cF$ extends $\eta\cup \nu$, 
\item \label{I.CR.m} $m< 3^{-|I_{<n}|}\log_2 k$
\pushcounter
\end{enumerate}
For $\fc$ and $\fd$ in   $\CRx$ let $\fd\in \Sigx(\fc)$ if the following conditions
hold
\begin{enumerate}
\popcounter
\item $n_\fd=n_\fc$, 
\item $\eta_\fc\subseteq \eta_\fd$, 
\item $k_\fc\geq k_\fd$,
\item $\cF_\fc\supseteq \cF_\fd$, 
\item $m_\fc\leq m_\fd$. 
\pushcounter
\end{enumerate}
For $\fc\in \CRx$ we define the following
\begin{enumerate}
\popcounter
\item $\nor_0(\fc)=\lfloor 3^{-|I_{<n}|}\log_2k \rfloor$
\item $\nor(\fc)=\nor_0(\fc)-m$,  
\item\label{c.pos}  $\pos(\fc)=\cF$. 
\pushcounter
\end{enumerate}
Therefore $\fc$ is a finite `forcing notion' that `adds' a function from $I_{J_n}$ into $\{0,1\}$. 
Its `working part' (or, the already decided part of the `generic' function) is $\eta_\fc$
and $\cF_\fc$ is the set of `possibilities' for the generic function (thus the redundant 
notation
\eqref{c.pos} included here for the purpose of compatibility with \cite{RoSh:470}). 
The `norm' $\nor(\fc)$ provides a lower bound on the amount of freedom
allowed by $\fc$ in determining the generic function.

For a relevant parameter $\bfx$ we now define the creature forcing $\bbQ=\bbQ_\bfx$. 
Let $\bfH(n)=2^k$, where $k=I_{J_n}$. This is the number of `generics' for $\fc\in \Sigx$ with $n_\fc=n$. 
Also let
\[
\textstyle\phi_{\bfH}(j)=|\prod_{i<j} \bfH(i)|. 
\]
Fix a function $f\colon \omega\times\omega\to\omega$ which satisfies the following
conditions for all $k$ and $l$ in $\omega$: 
\begin{enumerate}
\popcounter
\item $f(k,l)\leq f(k,l+1)$, 
\item $f(k,l)<f(k+1,l)$, 
\item $\phi_{\bfH}(l)(f(k,l)+\phi_{\bfH}(l)+2)<f(k+1,l)$.
\pushcounter
\end{enumerate}
We say such $f $ is  \emph{$\bfH$-fast} (cf. \cite[Definition~1.1.12]{RoSh:470}). 
 
 We now let $\bbQ_{\bfx}$ be $\bbQ_f(\CRx,\Sigx)$, as in \cite[Definition~1.1.10 (f)]{RoSh:470}. 
 This means that a typical condition in $\bbQ$ is a triple 
 \[
 p=(f_p,i(p),\bar \fc(p))
 \] 
such that (we drop subscript $p$ when convenient)
\begin{enumerate}
\popcounter
\item $f\colon I_{<i(p)}\to \{0,1\}$ for some $i(p)\in \omega$,  
\item $\bar\fc(p)=\langle \fc(p,j): j\geq i(p)\rangle$
\item Each $\fc(p,j)$ is in $\CRx$ and satisfies $n_{\fc(p,j)}=j$, 
\item \label{p.nor} with 
 $m_j=\min( I_{\min (J_j)})$
(cf. \cite[Definition~1.1.10(f)]{RoSh:470}) we have
\[
(\forall k)(\forall^\infty j)(\nor(\fc(p,j))>f(k,m_j).
\]
\pushcounter
\end{enumerate}
We let $q\leq p$ (where $q$ is a condition stronger than $p$) if the following conditions
are satisfied
\begin{enumerate}
\popcounter
\item $f_p\subseteq f_q$, 
\item $\fc(q,j)\in \Sigx(\fc(p,j))$ for $j\geq i(q)$, 
\item $f_q\rs I_j\in \pos(\fc(p,j))$ for $j\in [i(p),i(q))$. 
\end{enumerate} 
The idea is that $\bbQ_\bfx$ adds a function $\dot f$ from $\omega$ into $\{0,1\}$. 
A condition $p=(f_p, i(p), \bar\fc(p))$ decides that $\dot f$ extends $f_p$ as well as 
$f_{\fc(p,j)}$ for all $j\geq i(p)$. Also, $\pos(\fc(p,j))$ is the set of possibilities for
the restriction of $\dot f$ to $I_{J_j}$. The `norms on possibilities' 
condition \eqref{p.nor} affects the `rate' at which decisions are being made. 

 Experts may want to take note that with our creating pair $(\CRx,\Sigx)$ there is no difference between 
 $\bbQ_f$ and $\bbQ_f^*$ (cf. \cite[Definition~1.2.6]{RoSh:470}) since the intervals $J_n$ form a partition 
 of $\omega$. This should be noted since the results from
  \cite{RoSh:470} quoted below apply to $\bbQ_f^*$ and not $\bbQ_f$ in general.

\subsection{Properties of $\bbQ_\bfx$}
We shall need several results from \cite{RoSh:470} where the class of forcings
to which $\bbQ_\bfx$ belongs was introduced and studied.

\begin{lemma}\label{L.CR.0} The forcing notion $\bbQ_\bfx$ is nonempty and nonatomic. 
\end{lemma} 

Given $h\colon \omega\to \omega$ (typically increasing), we say that 
the creating pair $(\CRx,\Sigx)$ is \emph{$h$-big} (\cite[Definition~2.2.1]{RoSh:470}) 
if for each $\fc\in \CRx$  such that $\nor(\fc)>1$  and $\chi\colon \pos(\fc)\to h(n(\fc))$ 
there is $\fd\in\Sigx(\fc)$ such that $\nor(\fd)\geq \nor(\fc)-1$ and $\chi\rs\pos(\fd)$ is 
constant. We need only $h$-bigness in the case when $h(n)=2$ for all $n$.

\begin{lemma}\label{L.CR.1} If $h(n)=3^{|I_{<n}|}$ then 
the pair $(\CRx,\Sigx)$ is $h$-big. 
\end{lemma} 

\begin{proof} Fix $\fc=(n,u,\eta,\cF,m,k)\in \CRx$ such that $\nor (\fc)=\lfloor 2^{-|I_{<n}|}\log_2 k\rfloor -m>0$
and a partition $\cF=\bigcup_{j<r} \cF$, with $r=3^{|I_{<n}|}$. 
We need to find $\fd\in \Sigx(\fc)$ such that $\nor(\fd)\geq \nor(\fc)-1$ and 
$\cF_{\fd}\subseteq \cF_j$ for some $j$. 

Since $\nor(\fc)=\lfloor r \log_2 k\rfloor -m>0$, 
we have that $\log_2 k\geq r$ and therefore  
$k'=\lceil r k\rceil>0$.

We shall find $\fd$ of the form $(n,v, \zeta, \cF_j, m, k')$ for appropriate 
$v, \zeta$ and $j<r$. Note that $\nor(\fd)=\lfloor r\log_2 k'\rfloor -m =\nor(\fc)-1$. 
We shall try to find $u_j$ and $\eta_j\colon u_j\to \{0,1\}$ for $j<r$ as follows. 
If $\fd_0=(n,u,\eta, \cF_0, m,k')\in \Sigx(\fc)$, we let $\fd=\fd_0$ and stop. 
Otherwise, there are  $v_0\subseteq J_n\setminus u$ and $\zeta_0\colon v_0\to \{0,1\}$
such that $\eta\cup \zeta_0$ has no extension in $\cF_0$. 
Let $u_1=u\cup v_0$ and $\eta_1=\eta\cup \zeta_0$. 
If $\fd_1=(n,u_1,\eta_1,\cF_1, m,k')\in \Sigx(\fc)$, we let $\fd=\fd_1$ and stop. 
Otherwise, there are $v_1\subseteq J_n\setminus u_1$ and $\zeta_1\colon v_1\to \{0,1\}$
such that $\eta_1\cup \zeta_1$ has no extension in~$\cF_1$. 
Let $u_2=u_1\cup v_1$ and $\eta_2=\eta\cup \zeta_1$. 
Proceeding in this way, for $j<r$  we construct $v_j,u_j,\zeta_j$ 
and $\eta_j$ such that $\eta_j$ has no extension 
in $\cF_j$ or we find $\fd_j$ witnessing $r$-bigness of $\fc$. 
If $u_j$ and $\eta_j$ are constructed for $j<r-1$, then $v=\bigcup_{j<r} v_j$ has
cardinality $r k'=k$ and $\nu=\bigcup_{j<r} \zeta_j$ has no extension in~$\cF$. 
But this contradicts the assumption \eqref{I.CR.k} on $\fc$. 
Therefore one of $\fd_j$ is as required. 
\end{proof}

A creating pair $(\CRx,\Sigx)$ has the \emph{halving property} (\cite[Definition~2.2.7]{RoSh:470}) 
if for each $\fc\in \CRx$ such that $\nor(\fc)>0$ there is $\fd\in \Sigx(\fc)$ (usually 
denoted $\half(\fc)$) such that 
\begin{enumerate}
\item $\nor(\fd)\geq \frac 12 \nor(\fc)$, 
\item If in addition $\nor(\fc)\geq 2$ then for each $\fd_1\in \Sigx(\fd)$ such that 
$\nor(\fd_1)>0$ there is $\fc_1\in \Sigx(\fc)$ such that 
$\nor(\fc_1)\geq \frac 12 \nor(\fc)$ and $\pos(\fc_1)\subseteq \pos(\fd_1)$. 
\end{enumerate}

\begin{lemma}\label{L.CR.2}  
The pair $(\CRx,\Sigx)$ has the halving property. 
\end{lemma} 

\begin{proof} 
$\fc=(n,u,\eta,\cF,m,k)\in \CRx$ such that $\nor (\fc)=2^{-|I_{<n}|}-m>0$. 
Write $r=3^{-|I_{<n}|}$. Since $m<rk$ by \eqref{I.CR.m} we have that 
  $m_\fd=\frac 12 (r k +m)$ satisfies $m'<rk$  and therefore 
$\fd=(n,u,\eta,\cF,m_\fd,k)$ is in $\Sigx(\fc)$. 

Now let us assume $\nor(\fc)\geq 2$ since otherwise there is nothing left to do. 
Assume $\fd_1=(n,u_1,\eta_1,\cF_1,k_1,m_1)\in \Sigx(\fd)$ is such that $\nor(\fe)>0$. 
Note that  $\nor(\fd_1)=r\log_2(k_1)-m_1$, $m_1\geq m_\fd$ and $k_1\leq k_\fd=k$. 

Let $\fc_1=(n,u_1,\eta_1,\cF_1,k_1,m)$. 
Then 
\[
\nor(\fc_1)=\lfloor r \log_2 k_1 \rfloor -m
=\nor(\fd_1)-m+m_1\geq \frac 12 \nor(\fc_1)
\]
as required. 
\end{proof} 

Recall that a forcing notion $\bbP$ is \emph{$\omega^\omega$-bounding} 
if for every name $\dot f$ for an element of $\omega^\omega$ and 
every $p\in \bbP$ there are $q\leq p$ and $g\in \omega^\omega$
such that $q\forces \dot f(n)\leq \check g(n)$ for all $n$.

\begin{prop} \label{P.CR} Forcing notion $\bbQ_\bfx$ is proper, $\omega^\omega$-bounding,  
and both the ordering and the incomparability relation on $\bbQ_\bfx$ are  Borel. 
\end{prop} 

\begin{proof} 
In addition to bigness and halving properties of $\bbQ_\bfx$ proved in two lemmas above, 
we note that 
this forcing is finitary (i.e,, each $\CRx$ is finite) and simple (i.e., $\Sigx(S)$ is not defined
for $S\subseteq \CRx$ that contains more than one element). 
By \cite[Corollary~2.2.12 and Corollary~3.1.2]{RoSh:470}, or rather by 
\cite[Theorem~2.2.11]{RoSh:470}, 
it is proper and $\omega^\omega$-bounding. 

It is clear that $\leq_{\bbQ_\bfx}$ is Borel. 
We check the  remaining fact, that the relation $\perp_{\bbQ_\bfx}$ is Borel. 
Function $g\colon (\bbQ_\bfx)^2\to  \omega^ \omega$ defined by 
\[
g(p,q)(n)=\max\{\nor(\fd): \fd\in \Sigx(\fc(p,n))\cap \Sigx(\fc(q,n))\}
\]
(with $\max\emptyset=0$) is continuous. Since $p$ and $q$ are compatible if and only if 
$g(p,q)$ satisfies the largeness requirement \eqref{p.nor}, 
the incompatibility relation is Borel. 
\end{proof} 

\section{Forcing Iteration}
In this long section we analyze properties of forcings used in our proof. 

\subsection{Fusions and continuous reading of names in the iteration}
\label{S.CRN}
A crucial property of the forcing iteration used in our proof is that it has the 
continuous reading of names (by $\bbR$ we will usually mean $\cP(\omega)$). 

\begin{definition} \label{Def.crn} Consider 
a countable support 
forcing iteration  $(\bbP_\xi, \dot\bbQ_\eta: \xi\leq \kappa, \eta<\kappa)$ such that 
each $\dot\bbQ_\eta$ is a ground-model Suslin forcing notion which 
adds a generic real $\dot g_\xi$. Such an iteration   
has \emph{continuous reading of names} if  for every $\bbP_\kappa$-name $\dot x$ for a new real 
the set of conditions $p$ such that there exists countable $S\subseteq \kappa$, 
compact $F\subseteq \bbR^S$, and continuous $h\colon F\to \bbR$ such that 
\[
p\forces\text{``} \langle \dot g_\xi: \xi\in S\rangle \in F\text{ and } \dot x=h(\langle \dot g_\xi: \xi\in S\rangle)\text{''}
\]
is dense. 
\end{definition}

For iterations of proper forcing notions of the form $P_I$ where $I$ is  a $\bSigma^1_1$ 
on $\bPi^1_1$ 
$\sigma$-ideal of Borel sets (see \cite{Zap:Forcing}) continuous reading of names follows
from posets being $\omega^\omega$-bounding. This is a beautiful result of 
Zapletal (\cite[Theorem 3.10.19 and Theorem 6.3.16]{Zap:Forcing}). 
While many proper forcings adding a real are equivalent to ones 
of the form $P_I$ (see  \cite{Zap:Descriptive}. 
  and \cite{Zap:Forcing}), 
this  
unfortunately does not necessarily apply to 
 creature forcings as used in our proof (see \cite[\S 3]{Kell:Non-elementary}). 

Nevertheless, continuous reading of names in our iteration is   
 a special case of the results in \cite{Sh:630}. For 
 convenience of the reader we shall include a proof of this fact in Proposition~\ref{P.crn} below.   

We shall define a sequence of finer orderings on 
$\bbQ_{\bfx}$ (see \cite[Definition~1.2.11 (5)]{RoSh:470}). 
For $p\in \Qx$ and $j\in \bbN$ let 
\[
\chi(p,j)=\{r\in \Qx: r\leq p, f_r=f_q, \text{  and $\fc(r,i)=\fc(p,i)$ for all $i\leq j$}\}. 
\] 
For $p$ and $q$ in $\bbQ_{\bfx}$  and $n\geq 1$ write
\begin{enumerate}
\item $p\leq_0 q$ if $p\leq q$ and $f_p=f_q$, 
\item $p\leq_n q$ if 
\begin{enumerate}
\item $p\leq_0 q$, and with  $m_i=\min(I_{\min J_i})$  and 
\[
k=\min\{i : \nor(\fc(q,i))>f(n, m_i)\}
\]
we have 
\item $q\in \chi(p,k)$,  and  
\item $\nor(\fc(p,i))\geq f(n,m_i)$ for all $i$ such that $\fc(p,i)\neq \fc(q,i)$. 
\end{enumerate}
\end{enumerate}
We say that $p\in \bbQ_{\bfx}$ 
\emph{essentially decides} a name for an ordinal $\dot m$ if there exists~$j$ such that 
 every $q\in \chi(p,j)$ decides $\dot m$. 

By Theorem 2.2.11, if $\dot m$ is a name for an ordinal and $p\in \bbQ_{\bfx}$ then 
for every $n\in \omega$ there exists $q\leq_n p$ which essentially decides $\dot m$ (of course this 
is behind the proof of  Proposition~\ref{P.CR},   modulo standard fusion arguments). 

Let us now consider $\calR$, the standard poset for adding a random real. 
Conditions are compact subsets of $\cP(\omega)$ of positive Haar measure $\mu$ and 
the ordering is reverse inclusion. For $n\in \omega$ define a finer ordering on $\calR$ 
by 
$q\leq_n p$ if $q\leq p$ and $\mu(q)\geq (1-2^{-n-1})\mu(p)$. 
We say that 
$q\in \calR$ essentially decides $\dot m$ if there exists $j$ such that $q\cap [s]$ decides 
$\dot m$ for every $s\in 2^j$ such that $q\cap [s]\in \calR$. 
The inner regularity of $\mu$ implies that 
 for every name $\dot m$ for an ordinal, every $p\in \calR$ and every $n$ 
there exists $q\leq_n p$ which essentially decides $\dot m$.

In the following proposition we 
assume $\iteration$ 
is a countable support 
iteration  such that 
each $\dot\bbQ_\eta$ is  either some $\bbQ_{\bfx}$ or $\calR$,  
 and that in addition 
the maximal condition of $\bbP_\eta$ decides whether  $\dot \bbQ_\eta$ 
is $\calR$ or $\Qx$, and in the latter case it also decides  $x$, for all~$\eta$.

\begin{prop} \label{P.crn} An iteration $\iteration$ as in the above paragraph
has the continuous reading of names. 
 \end{prop}
 
 \begin{proof}  Since $\bbP_\kappa$ is a 
  countable support iteration 
 of proper, $\omega^\omega$-bounding forcing notions, by 
  \cite{Sh:f} the iteration is proper and $\omega^\omega$-bounding. 
 
 Let $\dot g$ be a name for an element of $\omega^\omega$. By the above and by working below a condition, 
 we may assume  that there exists $h\in \omega^\omega$ such that 
  $\forces_{\bbP} \dot g\leq \check h$. 
  Choose a countable elementary submodel $M$ of $H_{(2^\kappa)^+}$ containing everything
  relevant and let $F_j$, for $j\in \omega$, be an increasing sequence of  finite subsets of $M\cap \kappa$ with union equal to $M\cap \kappa$. 
  By using order $\leq_n$ in $\bbQ_{\bfx}$ and in $\calR$ introduced above, we can construct a
  fusion sequence $p_n$ such that for every $n$ and every $\eta\in F_n$ we have  
  \begin{enumerate}
  \item  $p_n\in M$, 
  \item $p_{n+1}\rs \eta\forces p_{n+1}(\eta)\leq_n p_n$, 
\item  \label{I.digits}  $p_n$ decides the first $n$ digits of $\dot g$ (we can do this 
since $\dot g\leq h$ implies there are only finitely many possibilities). 
\item $p_n\rs \eta$ decides $\fc(p_n(\eta),j)$ for $j\leq n$ if $\dot \bbQ_\eta=\Qx$ for some $x$
or decides $\{s\in 2^n: p_n(\eta)\cap [s]\neq \emptyset\}$ if $\dot \bbQ_\eta=\calR$.  
\end{enumerate}
Then for every $\eta\in M\cap \kappa$ and $n$ large enough 
the condition $p_{n+1}\rs \eta$ for $n\in \bbN$
forces that $p_{n+1}(\eta)\leq_n p_n(\eta)$. Therefore we can define a fusion 
$p$ of sequence $p_n$. Since $p_n\in M$ for all $n$ we have that the 
support of $p$ is included in $S=M\cap \kappa$. Let $F$ be the closed 
subset of $\cP(\bbN)^S$ whose complement is the union of all basic open $U\subseteq \cP(\bbN)^S$
such that $p\forces \dot x\notin U$. 
By \eqref{I.digits} there is a continuous function $h\colon F\to \omega^\omega$ such that 
$p\forces h(\langle \dot g_\xi: \xi\in S\rangle)=\dot x$. 
 \end{proof} 

\subsection{Subiterations and complexity estimates}\label{S.Subiterations}
Assume $\iteration$ is an iteration as in Proposition~\ref{P.crn}. 
Then for every subset $S\subseteq \kappa$ we have a well-defined 
subiteration 
\[
\bbP_S=(\bbP_\xi, \dot\bbQ_\eta: \xi\in S, \eta\in S). 
\]
We shall write $\forces_S$ instead of $\forces_{\bbP_S}$ and $\forces$ instead of $\forces_{\bbP_\kappa}$. 
In some specific 
situations  we have that $p\forces \phi$ is equivalent to $p\forces_S \phi$, where $S$ is the 
support of $p$.

The following result is a key to our proof. In the context of \cite{Zap:Forcing} 
much more can be said, but Zapletal's theory does not apply to the context of creature 
forcings (cf. paragraph after Definition~\ref{Def.crn}).

\begin{lemma} \label{L.PS} Assume $\bbP_\kappa$ is  
a countable support iteration of ground model  
$\omega^\omega$-bounding Suslin forcings. 
 Assume $B$ is a $\bPi^1_1$ set, $p\in \bbP_\kappa$,   $\dot x$ is a name for an element of $\cP(\bbN)$,     
$p \forces\dot x\in B$ and $\dot x$ is a $\bbP_{\supp(p)}$-name. 
Then  $p\forces_{\supp(p)} \dot x\in B$. 
\end{lemma}

\begin{proof} Let $S=\supp(p)$. Assume the contrary and find $q\leq p$ such that $q\forces_S \dot x\notin B$. Let $T$ be a tree whose projection is the complement of $B$ and let $\dot y$ be a name 
such that $q$ forces (in $\bbP_S$) that $(\dot y, \dot x)$ is a branch through $T$. 
Since $\bbP_S$ is an iteration of $\omega^\omega$-bounding forcings it is 
$\omega^\omega$-bounding (\cite{Sh:f}) we can assume (by extending $q$ if necessary) that 
$q\forces_S \dot y\leq \check h$ for $h\in \omega^\omega$.

Now choose a countable $M\prec H_\theta$ for a large enough $\theta$ so that $M$ contains
$\bbP_\kappa$, $q$, $\dot x$, $T$, $h$ and everything relevant. 
Let $G\subseteq \bbP_\kappa\cap M$ be an $M$-generic filter  containing $q$. 
Let 
$x=\intG(\dot x)$. The tree $T_x=\{s: (s,x\rs n)\in T$ for some $n\}$ is finitely 
branching (being included in $\{s: s(i)\leq h(i)$ for all $i<|s|\}$) 
and infinite.  
It therefore has an infinite branch by K\"onig's Lemma. This implies that $x\notin B$, contradicting the fact that 
$p\forces \dot x\in B$. 
\end{proof}

 Recall that a forcing notion is \emph{Suslin proper} if its underlying set is an analytic set of reals
 and both $\leq$ and $\perp$ are analytic relations.  
 The following lemma is well-known. 
 
 \begin{lemma} \label{L.Delta-1-2} 
Assume $\bbP$ is Suslin proper,  
$\dot x$ is a $\bbP$-name for a real, and  
$A\subseteq \bbR^2$ is Borel. 
Then for a dense set of conditions $p\in \bbP$ the set 
\[
\{a: p\forces (\check a,\dot x)\in A\}
\]
is $\bDelta^1_2$. 
\end{lemma} 

\begin{proof} Since $\bbP$ is proper the set of all    $p\in \bbP$ such that all 
antichains in  $\dot x$ are countable below $p$ is dense.  
For $a\subseteq \omega$ we now
 have that $ p\forces (\check a,\dot x)\in A$ 
 if there exists a countable well-founded 
model $M$ of $\ZFC$ containing everything relevant such that for every $M$-generic 
$G\subseteq M \cap \bbP$ with $p\in G$ we have that $F(a, \intG(\dot x))\in A$. 
This is a $\bSigma^1_2$ 
statement with $A$ as a parameter. 

Alternatively, $p\forces (\check a,\dot x)\in A$ if for every countable 
well-founded model $M$ of $\ZFC$
and every $M$-generic $G\subseteq M \cap \bbP$ with $p\in G$ 
we have that $F(a, \intG(\dot x))\in A$.
This is a $\bPi^1_2$-statement with $A$ as a parameter. 
\end{proof}

\begin{lemma} \label{L0} Assume $\iteration$ is as in Proposition~\ref{P.crn}. 
Assume $\dot x$ is a $\bbP$-name for a real,  
$A\subseteq \bbR$ is Borel and $g\colon\bbR^2\to \bbR$ is a Borel function. 
If $p\in \bbP$ is such that the name $\dot x$ is continuously read below $p$ then the set  
\[
\{a: p\forces g(\check a,\dot x)\in A\}
\]
is $\bDelta^1_2$. 
\end{lemma} 

\begin{proof} 
By Proposition~\ref{P.crn}, with  
$S=\supp(p)$ we have 
a compact 
$F\subseteq \cP(\bbN)^S$ and a continuous $h\colon F\to \cP(\omega)$ such that
$p\forces h(\langle \dot g_\xi: \xi\in S\rangle)=\dot x$. 

Lemma~\ref{L.PS} implies that $p\forces g(\check a,\dot x)\in A$
if and only if $p\forces_S g(\check a,\dot x)\in A$. 
Since $S$ is countable, by Lemma~\ref{L.Delta-1-2} the latter set is $\bDelta^1_2$. 
\end{proof}



\subsection{Reflection}\label{S.Ref}
Throughout this section we assume
 $(\bbP_\xi, \dot\bbQ_\eta: \xi\leq \kappa, \eta<\kappa)$ is a forcing iteration of proper forcings
 of cardinality $<\kappa$ in some model $M$ of a large enough fragment of ZFC. 
 We also assume 
 $G_\kappa\subseteq \bbP_\kappa$ is an $M$-generic filter and 
 let   $G\rs\xi$ denote $G\cap \bbP_\xi$. 
If $\dot A$ is a $\bbP_\kappa$-name for a set of reals we can consider it as a collection of 
nice names for reals. Furthermore, since $\bbP_\kappa$ is proper then we can identify $\dot A$ 
with a collection of pairs $(p,\dot x)$ where $p\in \bbP_\kappa$ and $\dot x$ is a name that involves only countable antichains below  $p$. The intention is that $p$ forces $\dot x$ is in $A$. 
With this convention we let  $\dot A\rs \xi$ denote the subcollection of $\dot A$ consisting only of those 
pairs $(p, \dot x)$ such that $p\in \bbP_\xi$ and $\dot x$ is a $\bbP_\xi$ name, 

The following `key triviality'  will be used repeatedly in proof of the main theorem. 
It ought to be  well-known but it does not seem to appear explicitly in the literature. 

\begin{prop}\label{P.sigma-1-2}
Assume $\kappa>\fc$ is a regular cardinal and  
\[
(\bbP_\xi, \dot\bbQ_\eta: \xi\leq \kappa, \eta<\kappa)
\]
  is a countable support iteration of proper forcings 
of cardinality $<\kappa$. 
Assume $\dot A$ is a $\bbP_\kappa$ name for a set of reals. Then the set of ordinals $\xi<\kappa$ such that
\[
(H(\aleph_1), \intGx\xi(\dot A\rs \xi))^{V[G\rs \xi]}\prec 
(H(\aleph_1), \intG(\dot A))^{V[G]}
\]
includes a club relative to $\{\xi<\kappa: \cf(\xi)\geq \omega_1\}$. 
\end{prop}  

 \begin{proof} Since each $\bbP_\xi$ is proper (\cite{Sh:f}), 
 no reals are added at stages of uncountable cofinality. 
 Therefore if $\cf(\eta)$ is uncountable then 
 $H(\aleph_1)^{V[G\rs\eta]}$ is the direct limit of
 $H(\aleph_1)^{V[G\rs\xi]}$ for $\xi<\eta$. 
 The assertion is now reduced to a basic fact from model theory: 
 club many substructures of $(H(\aleph_1), \intG(\dot A))^{V[G]}$ of cardinality $<\kappa$
 are elementary submodels. 
\end{proof}

 
 \begin{definition} Using notation as in the beginning of \S\ref{S.Ref}
 we say that  a formula 
  $\phi(x,Y)$ (with parameters $x\in \bbR$ and $Y\subseteq \bbR$) \emph{reflects} 
(with respect to $\bbP_\kappa$) 
  if for every name $\dot a$ for a real and every name $\dot B$ for a set of reals the following
  are equivalent. 
  \begin{enumerate}
  \item 
  $V[G]\models \phi(\dot a, \dot B)$, and
  \item There is a club $\bfC\subseteq \kappa$ such that 
  for all $\xi\in \bfC$ with $\cf(\xi)\geq \omega_1$ we have 
  $V[G\rs \xi]\models \phi(\dot a, \dot B\rs \xi)$.  
\end{enumerate}
\end{definition}
 
 \begin{coro}\label{C.sigma-1-2} 
 Let $(\bbP_\xi, \dot\bbQ_\eta: \xi\leq \kappa, \eta<\kappa)$ be a countable support iteration 
 of proper forcings of cardinality $<\kappa$. 
 Assume $\dot\cI$ and $\dot\cJ$ are $\bbP_\kappa$-names  
for Borel ideals on $\omega$ and  $\dot\Phi$ is a $\bbP_\kappa$-name for an isomorphism between their 
quotients. 
\begin{enumerate}
\item for every name $\dot a$ for a real the statement $\dot a\in \Triv^1_{\dot \Phi}$ reflects. 
 \item For $0\leq j\leq 2$ the statement ``$\Triv^j_{\dot\Phi}$ is meager'' reflects. 
\item For every  $\bbP_\kappa$-name $\dot I$ for a partition of $\omega$ into finite sets
the statement $\dot \cI\subseteq \calH(\dot I)$ reflects. 
\end{enumerate}
\end{coro}
\begin{proof} Since the pertinent statements are projective with the interpretation of 
$\dot\Phi$ as a parameter, each of the  assertions is a consequence of Proposition~\ref{P.sigma-1-2}.  
\end{proof} 

\subsection{Random reals} 
 
 We identify  $\cP(\omega)$ with $2^\omega$ and with $(\bbZ/2\bbZ)^\omega$ 
 and equip it with  the corresponding Haar measure.  
 The following lemma will be instrumental in the proof of one of our key lemmas, Lemma~\ref{L3}. 
 
 \begin{lemma}\label{L.x1} 
Assume $\cJ$ is a  Borel ideal and
$f$ and $g$ are continuous functions such that 
each one of them is a representation of a homomorphism 
from $\cP(\omega)$ into $\cP(\omega)/\cJ$. If the set 
\[
\Delta_{f,g,\cJ}=\{c\subseteq \omega: f(c)\neq^\cJ g(c)\}
\]
is null 
then it is empty. 
\end{lemma}

 \begin{proof} By the inner regularity of Haar measure we can find a compact set $K$ 
 disjoint from $\Delta_{f,g,\cJ}$ of measure $>1/2$. 
 Fix any $c\subseteq \omega$. The sets $K$ and $K\underline\Delta c=\{b\Delta c: b\in K\}$
 both have measure $>1/2$ and therefore we can find $b\in K$ such that $b\Delta c\in K$. 
 But then 
 \[
 f(c)=^\cJ f(c\Delta b)\Delta f(b)=^\cJ g(c\Delta b)\Delta g(b)=^\cJ g(c)
 \]
 completing the proof. 
 \end{proof} 
 
 In the following $\calR$ denotes the forcing for adding a random real and $\dot x$ is 
 the canonical $\calR$-name for the random real. 

 \begin{coro} 
 \label{C.x1} 
 Assume  $\cJ$ is a  Borel ideal and
$f$ and $g$ are continuous functions such that 
each is a representation of a homomorphism from $\cP(\omega)$ into $\cP(\omega)/\cJ$. 
Furthermore assume 
$\calR$ forces $f(\dot x)=^\cJ g(\dot x)$. 
Then $f(c)=^\cJ g(c)$ for all $c\subseteq \omega$. 
\end{coro}  

\begin{proof} It will suffice to show that the assumptions of Lemma~\ref{L.x1} 
are satisfied. This is a standard fact but we include the details. 
Since the set $\Delta_{f,g,\cJ}$ is Borel, if it is not null then there 
exists a compact set $K\subseteq \Delta_{f,g,\cJ}$ of positive measure. 
If $M$ is a countable transitive model of a large enough fragment of ZFC 
containing codes for $K,f,g$, and $\cJ$ and
$x\in K$ is a random real over $M$, then $M[x]\models f(x)=^\cJ g(x)$ 
by the assumption on $f$ and $g$. However, this is a $\Delta^1_1$ statement
and is therefore true in $V$. But $x\in \Delta_{f,g,\cJ}$ and therefore $f(x)\not=^\cJ g(x)$, 
a contradiction. 
\end{proof}

\subsection{Trivializing automorphisms locally and globally}\label{S.L2} 
Ever since the second author's proof that all automorphisms of $\cP(\bbN)/FiN$ are 
trivial in an oracle-cc forcing extensions (\cite{Sh:b}), every proof that 
automorphisms of a similar quotient structure proceeds in (at least) two stages. 
In the first stage one proves that the automorphism is `locally trivial' and in the second
stage local trivialities are pieced together into a single continuous representation (see 
e.g., \cite[\S3]{Fa:AQ}). 
The present proof is no exception. 

Throughout this subsection we assume 
\[
(\bbP_\xi, \dot\bbQ_\eta: \xi\leq \fc^+, \eta<\fc^+)
\]
is as in Proposition~\ref{P.crn}. 
Therefore it 
is a countable support 
iteration  such that 
each $\dot\bbQ_\eta$ is  either some $\bbQ_{\bfx}$ or $\calR$,  
 and that in addition 
the maximal condition of $\bbP_\eta$ decides whether  $\dot \bbQ_\eta$ 
is $\calR$ or $\Qx$, and in the latter case it also decides  $x$, for all~$\eta$. 
We shall write 
$p\forces_\xi \phi$ instead of~$p\forces_{\bbP_\xi} \phi$. 


\begin{lemma} \label{L2.1} With $\iteration$ as above, 
assume that for  every  partition $I$ of $\omega$ into finite intervals
the set 
\[
\{\xi<\fc^+\:  \forces_\xi \text{``$ \bbQ_\xi$ captures $I$ and $\cf(\xi)$ is uncountable''}\}
\]
 is stationary. 
Then every homomorphism between quotients over Borel ideals is locally $\bDelta^1_2$ (see \S\ref{S.Defs}). 
\end{lemma} 

\begin{proof} Fix a name $\dot \Phi$ for a homomorphism between quotients over Borel 
ideals $\cI$ and $\cJ$. By moving to an intermediate
forcing extension containing relevant Borel codes, 
we may assume the ideals $\cI$ and $\cJ$ are in the ground model. 
Let $G\subseteq \bbP_{\fc^+}$ be a generic filter. 

Assume $\Triv^2_{\intG(\dot \Phi)}$ is  meager in $V[G]$ with a witnessing partition  $\intG(\dot I)$ 
(cf. the discussion before Corollary~\ref{C.sigma-1-2}). 
By Corollary~\ref{C.sigma-1-2} the set of $\xi<\fc^+$ of uncountable cofinality 
such that $\intGx\xi(\dot I)$ witnesses $\Triv^2_{\intGx\xi(\dot \Phi\rs \xi)}$ is meager in $V[G\rs \xi]$
includes a relative club. 

Since the iteration of proper $\omega^\omega$-bounding forcings is proper and $\omega^\omega$-bounding (\cite{Sh:f}) 
the 
forcing is $\omega^\omega$-bounding, we may assume $\intG(\dot I)$ is a ground-model 
partition, $\vec I=(I_n: n\in \omega)$
By our assumption, there is a stationary set $\bfS$ of ordinals of uncountable cofinality 
such that for all $\eta\in \bfS$
we have
\begin{enumerate}
\item $\forces_\eta$``$ \dot \bbQ_\xi$ adds a real $\dot x$ that captures $\vec I$''.
\end{enumerate}
Fix $\eta\in \bfS$ for a moment. By going to the intermediate extension we may assume $\eta=0$. 
Let $\dot y$ be a name for a subset of $\omega$ 
such that 
\[
[\dot y]_{\cJ}=\Phi([\dot x]_{\cI}). 
\]
By the continuous reading of names (Proposition~\ref{P.crn})  we can find condition $p$ with 
support $S$ containing $0$, compact 
$F\subseteq \cP(\bbN)^S$ and continuous $h\colon F\to \cP(\omega)$ such that
$p\forces h(\langle \dot g_\xi: \xi\in S\rangle)=\dot y$. Note that $\dot x$ is equal to $\dot g_0$
hence it is ``continuously read.''

Since $\bbQ_0$ captures $I$ we can find an infinite $d$ such that with $a=I_d$ 
for every $b\subseteq a$ condition $p_b\leq p$ forces $\dot x\cap a=b$. 
Also, $\supp(p_b)=\supp(p)$ 
and
(by the definition of $\Qx$) the map $b\mapsto p_b$ is continuous. 

By the choice of $\dot y$, with  $c=\Phi_*(a)$  
by Lemma~\ref{L.PS} we have that 
\[
p_b\forces_S \dot y\cap c=^{\cJ} \Phi_*(b).
\]
By Lemma~\ref{L.Delta-1-2} the set 
\[
\{(b,e): b\subseteq a, e\subseteq c, e=^{\cJ}\Phi_*(b)\}
\]
is $\bDelta^1_2$. 

Therefore $a$ 
and $\dot \bbQ_\xi$ satisfy the assumptions of Lemma~\ref{L0} and 
 in $V[G\rs\xi]$ the restriction of $\intGx\xi(\dot\Phi\rs \xi)$  to $\cP(a)/\cI$
is $\bDelta^1_2$, contradicting our assumption. 
Since assuming $\Triv^2_{\intG(\dot\Phi)}$ was meager lead to a contradiction, this concludes the proof. 
\end{proof} 

\begin{definition}
Assume $\bbP$ is a forcing notion and $\dot\Phi$ is a $\bbP$-name for an isomorphism
between quotients over Borel ideals $\cI$ and $\cJ$ 
which extends ground-model isomorphism  $\Phi$ 
between these quotients. 
We say that $\dot\Phi$ is \emph{$\bbP$-absolutely locally topologically trivial} if the following 
apply (in order to avoid futile discussion we assume $\bbP$ is $\omega^\omega$-bounding):
\begin{enumerate}
\item $\Phi$ is locally topologically trivial, 
\item  $\bbP$ forces that the 
 continuous witnesses of local topological triviality of $\Phi$ 
 witness local topological triviality of $\dot\Phi$. 
\end{enumerate}
\end{definition}

In order to justify this definition we note that 
this is not a consequence of the assumption that $\Phi$ is  locally topologically trivial
and $\dot\Phi$ is forced to be locally topologically trivial.  
 By a result of Stepr\=ans, there is  a $\sigma$-linked 
forcing notion such that a trivial automorphism of $\cP(\omega)/\Fin$ extends to a trivial automorphism, 
but the triviality is not implemented by the same function~(\cite{Ste:Autohomeomorphism}). 
Stepr\=ans used this to show that  there is a forcing iteration $\bbP_\kappa$ that forces 
Martin's Axiom and the existence of a nontrivial  automorphism $\Phi$ of $\cP(\omega)/\Fin$
that is trivial in $V[G\rs \xi]$ for cofinally many~$\xi$.

 The following key lemma shows that in our forcing extension  
   local topological triviality is always witnessed by a  $\bPi^1_2$ set. 
 
\begin{lemma} \label{L3} 
Assume   $\iteration$  is as in the beginning \S\ref{S.L2}  such that 
$\bbQ_0$ is $\calR$. Also assume  $\dot\Phi$ is a $\bbP_\kappa$-name for a $\bbP_\kappa$-absolutely 
locally topologically trivial
 isomorphism  between quotients over Borel ideals $\cI$~and~$\cJ$. 
Then the set 
\[
\{(c,d): \Phi_*(c)=^\cJ d\}
\]
is $\bPi^1_2$. 
\end{lemma} 

\begin{proof} We have $\Phi\colon \cP(\omega)/\cI\to \cP(\omega)/\cJ$. 
Let $\dot x$ be the canonical $\bbQ_0$-name for the random real 
and let $\dot y$ be a $\bbP_\kappa$-name for the image of $\dot x$ by the extension of $\Phi$. 
By the continuous reading of names (Proposition~\ref{P.crn})  we can find condition $p$ with countable  
support $S$ containing $0$, compact 
$F\subseteq \cP(\bbN)^S$ and continuous $h\colon F\to \omega^\omega$ such that
$p\forces h(\langle \dot g_\xi: \xi\in S\rangle)=\dot y$. Again $\dot x$ is equal to $\dot g_0$
hence it is ``continuously read.''

Consider the set
$\cZ$ of all $(a,b,f,g)$ such that 
\begin{enumerate}
\item\label{I.Z1} $a$ and $b$  are   subsets of $\omega$. 
\item \label{I.Z2} $f\colon \cP(a)\to \cP(b)$ and $g\colon \cP(b)\to \cP(a)$ are continuous maps, 
\item\label{I.Z3} $f$ is a representation of a homomorphism from $\cP(a)/\cI$ into $\cP(b)/\cJ$, 
\item\label{I.Z4} $g$ is a representation of a homomorphism from $\cP(b)/\cJ$ into $\cP(a)/\cI$, 
\item \label{I.Z3+} $f(c)\in \cJ$ if and only if $c\in \cI$, 
\item\label{I.Z5} $f(g(c))=^{\cJ} c$ for all $c\subseteq b$, 
and  $g(f(c))=^{\cI} c$ for all $c\subseteq a$, 
\item \label{I.Z6} $p$ forces  that $f(\dot x\cap \check a)=^{\cJ}  \dot y\cap \check b$. 
\item \label{I.Z7} $p$ forces that $g(\dot y\cap \check b)=^{\cI}  \dot x\cap \check a$. 
\pushcounter
\end{enumerate}
Conditions \eqref{I.Z1} and \eqref{I.Z2} state that 
 $\cZ$ is a subset of the compact metric 
space $\cP(\omega)^2\times C(\cP(\omega), \cP(\omega))^2$, 
where $C(X,Y)$ denotes the compact metric 
space of continuous functions between compact metric spaces $X$ and $Y$.
Since   
 \eqref{I.Z3} states that 
\[
(\forall x\subseteq a)(\forall y\subseteq a) f(x\cup y)=^\cJ f(x)\cap f(y) 
\]
\[
(\forall x\subseteq a) f(a)\setminus f(x)=^\cJ f(a\setminus x)
\]
this is a  $\bPi^1_1$ condition, and similarly for  \eqref{I.Z4}. Similarly \eqref{I.Z3+}  and \eqref{I.Z5} 
are $\bPi^1_1$.   
Lemma~\ref{L0} implies that the remaining condition, \eqref{I.Z6}, 
is $\bDelta^1_2$ (recall that $\cJ$ was assumed to be  Borel). 
Therefore the set $\cZ$ is $\bDelta^1_2$. 
The set 
\[
\cK=\{a: (a,b,f,g)\in \cZ\text{ for some }(b,f,g)\}
\]
is easily seen to be an ideal that includes $\Triv_\Phi^1$.  Since $\Phi$ is locally topologically trivial 
it  is nonmeager. 

We shall now prove a few facts about the elements of $\cZ$. 

An  $(a,b,f,g)\in \cZ$ can be re-interpreted in the forcing extension, 
and in particular we identify function $f$ with the corresponding
continuous function. Properties \eqref{I.Z1}--\eqref{I.Z5} are $\bPi^1_1$ and therefore
still hold in the extension.  In particular $f$ is forced to be a representation of an isomorphism. 

For $a\in \cK$ let $f_a$ and $g_a$ denote functions such that $(a,b,f_a,g_a)\in \cZ$ for some $b$.  
For  $a\in \Triv^1_\Phi$ let $h_a\colon \cP(a)\to \cP(\omega)$ be a continuous
representation of $\Phi\rs a$. 
Let $\Phi_*$ denote a representation of the extension of  $\Phi$ in the forcing extension. 
\begin{enumerate}
\popcounter
\item \label{I.ZZ.0} If $a\in \cK$ then $f_a(c)=^{\cJ} \Phi_*(c)\cap b$ for all $c\subseteq a$. 
\pushcounter
\end{enumerate}
This is a consequence of Corollary~\ref{C.x1}, since \eqref{I.Z6} states that 
 $p$ forces 
 \[
 f_a(\dot x\cap \check a)=^{\cJ}\Phi_*(\dot x)\cap \check b.
\]
If $\Phi_*^{-1}$ denotes a representation of $\Phi^{-1}$ then 
by the same argument and~\eqref{I.Z7}  we have  
\[
g_a(d)=^{\cI} \Phi_*^{-1}(d)\cap a
\]
 for all $d\subseteq b$. 
\begin{enumerate}
\popcounter
\item \label{I.ZZ.0+} If $a\in \cK$ then $f_a(c)=^{\cJ} \Phi_*(c)$ for all $c\subseteq a$. 
\pushcounter
\end{enumerate}
Let $d=\Phi_*(a)\setminus b$ and  $c=\Phi_*^{-1}(d)$. Then $c\setminus a$ belongs to $\cI$.  
Also, with $c'=c\cap a$ we have  $f_a(c')=^{\cJ} \Phi_*(c)\cap b=d\cap b=\emptyset$. 
However, \eqref{I.Z5} and \eqref{I.Z4}  together with this imply 
\[
c=^{\cI}c'=^{\cI}g_a (f_a(c'))=^{\cI} g_a(\emptyset)=^{\cI}\emptyset. 
\]
Unraveling the definitions, we have that $\Phi_*^{-1}$ sends 
$\Phi_*(a)\setminus b$  to $\emptyset$ modulo $\cI$ and
therefore that $\Phi_*(a)=^{\cJ} b=f_a(a)$. By applying \eqref{I.ZZ.0} and Corollary~\ref{C.x1}, 
\eqref{I.ZZ.0+} follows. 
\begin{enumerate}
\popcounter
\item \label{I.ZZ.1} If $a\in \Triv_\Phi^1$ then $a\in \cK$ and $h_a(c)=^\cJ \Phi_*(c)=^\cJ f_a(c)$
for all $c\subseteq a$. 
\pushcounter
\end{enumerate}
That $a\in \cK$ is immediate from the definitions of $\cZ$ and $\cK$, and $h_a(c)=^\cJ \Phi_*(c)$
is immediate from $a\in \Triv_\Phi^1$ and the definition of $h_a$. 
The last equality, $\Phi_*(c)=^\cJ f_a(c)$ for all $c\subseteq a$, was proved in   \eqref{I.ZZ.0+}. 

Putting together \eqref{I.ZZ.0+} and \eqref{I.ZZ.1} we obtain that $\cK=\Triv^1_\Phi$ and
that $f_a$ witnesses $a\in \Triv^1_\Phi$ for every $a\in \cK$. 

\begin{enumerate}
\popcounter
\item \label{I.Z11} We have 
\[
\{(c,d) : \Phi_*(c)=^\cJ d\}= 
\{(c,d): (\forall (a,b,f,g)\in \cZ) f(c\cap a)=^\cJ b\cap d\}
\]
\pushcounter
\end{enumerate}
Take $(c,d)$ such that $\Phi_*(c)=^\cJ d$. Then for every $(a,b,f,g)\in \cZ$
we have $\Phi_*(c\cap a)=^\cJ f(c\cap a)$ by \eqref{I.ZZ.1} and \eqref{I.ZZ.0+}, and therefore 
$(c,d)$ belongs to the right-hand side set. 

Now take $(c,d)$ such that $\Phi_*(c)\Delta d$ is not in $\cJ$. 

Assume for a moment that $e=\Phi_*(c)\setminus d\notin \cJ$. 
Since $\Phi$ is an isomorphism, we can find  $a$  such that $\Phi_*(a)=^\cJ e$. 
We have that $a$ is $\cI$ positive. 
Since $\cK$ is nonmeager, by Lemma~\ref{L.y1} we can find  $a'\subseteq a$ 
such that $a'\in \cK\setminus \cI$. 
Then $f_{a'}(c\cap a')$ is $\cJ$-positive, included (modulo $\cJ$) 
in $e$, and disjoint (modulo $\cJ$) from $d$. Therefore $(a',f_{a'})$ witness that 
$(c,d)$ does not belong to the right-hand side of \eqref{I.Z11}.

We must therefore have  $e=d\setminus \Phi_*(c)\notin \cJ$ (there is no 
harm in denoting this set by  $e$, since the existence of the 
set denoted by $e$ earlier lead us to a contradiction). Applying the above argument
we can find $a'\in \cK$ such that $c\cap a'$ is $\cI$-positive, but its
image under $f_{a'}$ is included (modulo $\cJ$) 
in $d$ and disjoint (modulo $\cJ$) from $\Phi_*(c)$, which is again a contradiction.

By  \eqref{I.Z11} we have the required $\bPi^1_2$ definition of $\Phi$. 
\end{proof}

\section{Proofs}
\label{S.Proofs}

\subsection*{Proof  of Theorem~\ref{T0}} 

\label{S.ProofT1}
By \S\ref{S.CR}  for every partition $I$ of $\omega$ into finite intervals there is a
forcing notion of the form $\Qx$  that adds a real which captures $I$.   Each of these forcings is proper, real, has continuous reading of names and is $\omega^\omega$-bounding. 
Starting from a model of CH partition $\{\xi<\aleph_2: \cf(\xi)=\aleph_1\}$ 
into $\aleph_1$  stationary sets. 
Consider a countable support iteration 
$(\bbP_\xi, \dot\bbQ_\eta: \xi\leq \omega_2, \eta<\omega_2)$ 
of forcings of the form $\bbQ_\bfx$ and random reals such that for every $\dot I$ 
the set $\{\xi: \cf(\xi)=\omega_1$ and $\dot \bbQ_\xi$ is $\bbQ_\bfx\}$   is stationary 
and also $\{\xi: \cf(\xi)=\omega_1$ and $\dot \bbQ_\xi$ is the poset for adding a random real$\}$
is stationary. 

Since this forcing is a countable support  iteration of proper 
$\omega^\omega$-bounding forcings it is proper and $\omega^\omega$-bounding
(by \cite[\S VI.2.8(D)]{Sh:f}) and therefore $\fd=\aleph_1$ in the extension. 

Now fix names $\dot \cI$ an $\dot \cJ$ for Borel ideals and 
a name $\dot\Phi$ for an automorphism between Borel 
quotients $\cP(\omega)/\dot \cI$ and
$\cP(\omega)/\dot \cJ$. 
By Lemma~\ref{L2.1}, $\dot \Phi$ is forced to be locally $\bDelta^1_2$
and by Corollary~\ref{C.sigma-1-2} 
there is a stationary set $\bS$ of $\xi$  such that $\cf(\xi)=\omega_1$ such that 
$\dot\Phi\rs \xi$ is a $\bbP_\xi$ name for a a locally $\bDelta^1_2$-isomorphism,  
and $\dot\bbQ_\xi$ is the standard poset for adding a random real. 

By our assumption that all $\bSigma^1_2$ sets have the property of Baire and 
Lemma~\ref{L-1}, $\dot\Phi$ is forced to be locally topologically trivial. 
By Lemma~\ref{L3}, if $\xi\in\bS$ then   $\dot\Phi$ is  $\bPi^1_2$  in $V[G\rs\xi]$. 
Therefore  $\dot \Phi$ is  $\bPi^1_2$ in $V[G]$. 

Since our assumption that there exists a measurable cardinal  implies
 that we have $\bPi^1_2$-uniformization of this 
graph, $f\colon \cP(\omega)\to \cP(\omega)$
and all $\bPi^1_2$ sets have the Property of Baire, 
 $\Phi$ has a Baire-measurable representation. By a well-known fact 
(e.g.,  \cite[Lemma~1.3.2]{Fa:AQ}) 
$\Phi$ has a continuous representation.

In order to add $2^{\aleph_0}<2^{\aleph_1}$ to the conclusions,
 start from a model of CH and add $\kappa\geq\aleph_3$ of the so-called Cohen 
subsets of $\aleph_1$ to increase $2^{\aleph_1}$ to $\kappa$ while preserving CH.
More precisely, we force  with the poset of all countable partial functions $p\colon \aleph_3\times \aleph_1\to \{0,1\}$ ordered by the extension. 
Follow this by  the iteration $\bbP_{\aleph_2}$ of $\Qx$ and $\calR$ 
defined above. The above argument was not sensitive
to the value of $2^{\aleph_1}$ therefore all isomorphisms still have continuous representations. 
Finally, the iteration does not collapse $2^{\aleph_1}$ because a simple $\Delta$-system argument shows that it has $\aleph_2$-cc.

\subsection*{Proof of Theorem~\ref{T0.uB}} \label{S.uB}
Not much more remains to be said about this proof. 
Assume there exist class many Woodin cardinals, consider the very same forcing iteration 
as in  the proof of Theorem~\ref{T0} and fix names for universally Baire ideals $\dot\cI$ and $\dot\cJ$
as well as for an isomorphism $\dot\Phi$ between their quotients. Proofs of 
lemmas from \S\ref{S.LBP} show that the graph of $\dot\Phi$ is forced to be projective 
in $\dot\cI$ and $\dot\cJ$ and therefore universally Baire itself (see \cite{Lar:Stationary}). 
It can therefore be uniformized on a dense $G_\delta$ set by a continuous function, and 
therefore $\dot\Phi$ is forced to have a continuous representation.

\section{Concluding remarks} 
\label{S.ConcR}

As pointed out earlier, some of the ideas used here were present in the last section of \cite{ShSte:Martin}. 
However, in the latter only automorphisms of $\cP(\omega)/\Fin$ were considered and, more 
importantly, the model constructed there does have nontrivial automorphisms of $\cP(\omega)/\Fin$. 
This follows from  the very last paragraph of  
 \cite{ShSr:990}.

\begin{question} Are large cardinals necessary for the conclusion of Theorem~\ref{T0}? 
\end{question}

The answer is likely to be negative (as suggested by the anonymous referee) but it would 
 be nice to have a proof. 

Questions of whether isomorphisms with continuous representations are necessarily trivial 
and what can be said about triviality of homomorphisms as compared to isomorphisms (see 
\cite[Question~3.14.2]{Fa:AQ}) are 
as interesting as ever, but since we have no new information on these questions we shall move on. 
Problem~\ref{Pr.PFA}  reiterates one of the conjectures from \cite{Fa:Rigidity}, 
and a positive answer to \eqref{I.Pr.PFA.1} below 
 may require an extension of results about freezing gaps in Borel quotients 
 from~\cite{Fa:Luzin}. 

\begin{problem} \label{Pr.PFA} 
\begin{enumerate}
\item \label{I.Pr.PFA.1} Prove that PFA implies that all isomorphisms between quotients over Borel ideals 
have continuous representations.  
\item Prove that  all isomorphisms between quotients over Borel ideals 
have continuous representations in standard $\bbP_{\max}$ extension (\cite{Wo:Pmax2}, \cite{La:Forcing}).  
\end{enumerate}
\end{problem}

We end with two fairly ambitious questions. 
A positive answer to the  following   would be naturally conditioned on a large cardinal assumption
(see \cite{Fa:Absoluteness}). 

\begin{question} Is there a metatheorem analogous to Woodin's $\Sigma^2_1$ absoluteness 
theorem (\cite{Lar:Stationary}, \cite{Wo:Beyond}) and the $\Pi_2$-maximality of $\bbP_{\max}$ extension (\cite{Wo:Pmax2}, \cite{La:Forcing}), that provides a positive answer to Problem~\ref{Pr.PFA} (1) or (2)  automatically 
from Theorem~\ref{T0}? 
\end{question}

Let us temporarily abandon   Boolean algebras and 
briefly return to the general situation described in the introduction. 
Attempts to generalize these rigidity results to other categories 
were made, with limited success, in \cite{Fa:Liftings}. 
For example, quotient group $\prod_\omega \bbZ/2\bbZ/\bigoplus_{\omega}\bbZ/2\bbZ$
clearly has nontrivial automorphisms in ZFC. 
One should also mention the case of semilattices, when isomorphisms are locally trivial 
but not necessarily trivial  (\cite{Fa:Liftings}). 
On the other hand, 
PFA implies that all automorphisms of the Calkin algebra are trivial   (\cite{Fa:All}). 
 Note that 
`trivial' as defined here is equivalent to `inner' for automorphisms of the Calkin algebra, but this is not true for arbitrary corona algebras since in some cases the relevant multiplier algebra has outer automorphisms, unlike $\cB(H)$ (see \cite{CoFa:Automorphisms}, \cite{FaHa:Countable}).  

\begin{problem} In what categories can one prove  consistency of the
assertion that all isomorphisms between 
quotient structures based on standard Borel spaces are trivial? 
\end{problem}

\subsection{Groupwise Silver forcing} \label{S.CR.Silver} 
A simpler forcing notion that can be used in our proof in place of 
$\bbQ_\bfx$ defined above (\cite{Gha:FDD}). The `relevant parameter' is $\vec I=(I_n: n\in \omega)$, 
a partition of $\omega$ into finite intervals. 
Forcing $\bbS_{\vec I}$ consists of partial functions $f$ from a subset of $\omega$ into $\{0,1\}$
such that the domain of $f$ is disjoint from infinitely many of the $I_n$. Every condition $f$  
can be identified with the compact subset $p_f$ of $\cP(\omega)$ consisting of all functions 
extending $f$. Special cases of $\bbS_{\vec I}$ are Silver forcing (the case 
when $I_n=\{n\}$ for all~$n$) and `infinitely equal,'  or  EE, forcing (the case when $|I_n|=n$ for all $n$, see \cite[\S 7.4.C]{BarJu:Book}).

This is a suslin forcing and a fusion argument shows that it
is proper,  $\omega^\omega$-bounding and has continuous reading of names. 
Also, the proof that this forcing is $\omega^\omega$-bounding 
and proper are analogous to proofs of the corresponding facts for 
EE, \cite[Lemma~7.4.14]{BarJu:Book} and \cite[Lemma~7.4.12]{BarJu:Book}, respectively). 
Since this forcing is of the form $P_I$, Zapletal's results (\cite{Zap:Forcing}) make its analysis a bit more convenient. 
Proofs of these facts and  
applications of~$\bbS_{\vec I}$  to the rigidity of quotients appear in \cite{Gha:FDD}.

\subsection*{Acknowledgments} I.F.   is indebted to Jindra Zapletal for an 
eye-opening remark that shortened this paper for a couple of pages (see~\S\ref{S.CRN}) and 
for Lemma~\ref{L.PS}. 
He would also like to thank  Saeed Ghasemi for pointing out to a gap in the proof of Lemma~\ref{L3} in the original version  and finally 
to  Alan Dow, Michael Hru\v sak, Menachem Magidor, Arnie Miller, Juris Stepr\=ans, 
and the anonymous referee.


\begin{thebibliography}{10}

\bibitem{AlCoMac}
J.~L. Alperin, J.~Covington, and D.~Macpherson, \emph{Automorphisms of
  quotients of symmetric groups}, Ordered groups and infinite permutation
  groups, Math. Appl., vol. 354, Kluwer Acad. Publ., Dordrecht, 1996,
  pp.~231--247.

\bibitem{BarJu:Book}
T.~Bartoszynski and H.~Judah, \emph{Set theory: on the structure of the real
  line}, A.K. Peters, 1995.

\bibitem{BYBHU}
I.~Ben~Yaacov, A.~Berenstein, C.W. Henson, and A.~Usvyatsov, \emph{Model theory
  for metric structures}, Model Theory with Applications to Algebra and
  Analysis, Vol. II (Z.~Chatzidakis et~al., eds.), London Math. Soc. Lecture
  Notes Series, no. 350, Cambridge University Press, 2008, pp.~315--427.

\bibitem{ChaKe}
C.~C. Chang and H.~J. Keisler, \emph{Model theory}, third ed., Studies in Logic
  and the Foundations of Mathematics, vol.~73, North-Holland Publishing Co.,
  Amsterdam, 1990.

\bibitem{CoFa:Automorphisms}
S.~Coskey and I.~Farah, \emph{Automorphisms of corona algebras and group
  cohomology}, preprint, arXiv:1204.4839, 2012.

\bibitem{Fa:AQ}
I.~Farah, \emph{Analytic quotients: theory of liftings for quotients over
  analytic ideals on the integers}, Memoirs of the American Mathematical
  Society, vol. 148, no. 702, 2000.

\bibitem{Fa:Liftings}
\bysame, \emph{Liftings of homomorphisms between quotient structures and {U}lam
  stability}, Logic Colloquium '98, Lecture notes in logic, vol.~13, A.K.
  Peters, 2000, pp.~173--196.

\bibitem{Fa:CH}
\bysame, \emph{How many {B}oolean algebras {$\cP(\bbN)/\cI$} are there?},
  Illinois Journal of Mathematics \textbf{46} (2003), 999--1033.

\bibitem{Fa:Luzin}
\bysame, \emph{Luzin gaps}, Trans. Amer. Math. Soc. \textbf{356} (2004),
  2197--2239.

\bibitem{Fa:Rigidity}
\bysame, \emph{Rigidity conjectures}, Logic Colloquium 2000, Lect. Notes Log.,
  vol.~19, Assoc. Symbol. Logic, Urbana, IL, 2005, pp.~252--271.

\bibitem{Fa:All}
\bysame, \emph{All automorphisms of the {C}alkin algebra are inner}, Annals of
  Mathematics \textbf{173} (2011), 619--661.

\bibitem{Fa:Absoluteness}
\bysame, \emph{Absoluteness, truth, and quotients}, Proceedings of the {IMS}
  Workshop on Infinity and Truth (C.T. Chong et~al., eds.), World Scientific,
  to appear.

\bibitem{FaHa:Countable}
I.~Farah and B.~Hart, \emph{Countable saturation of corona algebras}, Comptes
  Rendus Mathematique (to appear), preprint, arXiv:1112.3898v1.

\bibitem{FeMaWo:uB}
Q.~Feng, M.~Magidor, and W.H. Woodin, \emph{Universally {B}aire sets of reals},
  Set Theory of the Continuum (H.~Judah, W.~Just, and W.H. Woodin, eds.),
  Springer-Verlag, 1992, pp.~203--242.

\bibitem{Gha:FDD}
S.~Ghasemi, \emph{*-homomorphisms of {FDD} algebras into corona algebras},
  preprint, York University, 2012.

\bibitem{Just:Repercussions}
W.~Just, \emph{Repercussions on a problem of {E}rd\"os and {U}lam about density
  ideals}, Canadian Journal of Mathematics \textbf{42} (1990), 902--914.

\bibitem{Just:Modification}
\bysame, \emph{A modification of {S}helah's oracle chain condition with
  applications}, Transactions of the American Mathematical Society \textbf{329}
  (1992), 325--341.

\bibitem{Just:WAT}
\bysame, \emph{A weak version of {AT} from {OCA}}, MSRI Publications
  \textbf{26} (1992), 281--291.

\bibitem{JustKr}
W.~Just and A.~Krawczyk, \emph{On certain {B}oolean algebras {${\mathcal
  P}(\omega)/I$}}, Transactions of the American Mathematical Society
  \textbf{285} (1984), 411--429.

\bibitem{Kana:Book}
A.~Kanamori, \emph{The higher infinite: large cardinals in set theory from
  their beginnings}, Perspectives in Mathematical Logic, Springer,
  Berlin--Heidelberg--New York, 1995.

\bibitem{KanRe:Ulam}
V.~Kanovei and M.~Reeken, \emph{On {U}lam's problem concerning the stability of
  approximate homomorphisms}, Tr. Mat. Inst. Steklova \textbf{231} (2000),
  249--283.

\bibitem{KanRe:New}
\bysame, \emph{New {R}adon--{N}ikodym ideals}, Mathematika \textbf{47} (2002),
  219--227.

\bibitem{Kell:Non-elementary}
J.~Kellner, \emph{Non-elementary proper forcing}, preprint, arXiv:0910.2132,
  2009.

\bibitem{Ku:Book}
K.~Kunen, \emph{Set theory: An introduction to independence proofs},
  North--Holland, 1980.

\bibitem{Lar:Stationary}
P.B. Larson, \emph{The stationary tower}, University Lecture Series, vol.~32,
  American Mathematical Society, Providence, RI, 2004, Notes on a course by
  W.H.Woodin.

\bibitem{La:Forcing}
\bysame, \emph{Forcing over models of determinacy}, Handbook of set theory.
  {V}ols. 1, 2, 3, Springer, Dordrecht, 2010, pp.~2121--2177.

\bibitem{MaSo:Basis}
D.A. Martin and R.M. Solovay, \emph{A basis theorem for {$\mathbf\Sigma^1_3$}
  sets of reals}, Ann. of Math. (2) \textbf{89} (1969), 138--159.

\bibitem{PhWe:Calkin}
N.C. Phillips and N.~Weaver, \emph{The {C}alkin algebra has outer
  automorphisms}, Duke Math. Journal \textbf{139} (2007), 185--202.

\bibitem{RoSh:470}
Andrzej Roslanowski and Saharon Shelah, \emph{Norms on possibilities i: forcing
  with trees and creatures}, Memoirs of the American Mathematical Society
  \textbf{141} (1999), xii + 167, math.LO/9807172.

\bibitem{Ru}
W.~Rudin, \emph{Homogeneity problems in the theory of \v{C}ech
  compactifications}, Duke Mathematics Journal \textbf{23} (1956), 409--419.

\bibitem{ShSte:PFA}
S.~Shelah and J.~Stepr{\=a}ns, \emph{{PFA} implies all automorphisms are
  trivial}, Proceedings of the American Mathematical Society \textbf{104}
  (1988), 1220--1225.

\bibitem{ShSte:Martin}
\bysame, \emph{Martin's axiom is consistent with the existence of nowhere
  trivial automorphisms}, Proc. Amer. Math. Soc. \textbf{130} (2002), no.~7,
  2097--2106.

\bibitem{Sh:b}
Saharon Shelah, \emph{Proper forcing}, Lecture Notes in Mathematics, vol. 940,
  Springer-Verlag, Berlin-New York, xxix+496 pp, 1982.

\bibitem{Sh:f}
\bysame, \emph{Proper and improper forcing}, Perspectives in Mathematical
  Logic, Springer, 1998.

\bibitem{Sh:630}
\bysame, \emph{Properness without elementaricity}, Journal of Applied Analysis
  \textbf{10} (2004), 168--289, math.LO/9712283.

\bibitem{ShSr:990}
Saharon Shelah and Juris Steprans, \emph{Nontrivial automorphisms of
  $\mathcal{P}(\mathbb{N})/ [\mathbb{N}]^{<\aleph_0}$ from variants of small
  dominating number}, preprint.

\bibitem{Ste:Autohomeomorphism}
Juris Stepr{\=a}ns, \emph{The autohomeomorphism group of the \v {C}ech-{S}tone
  compactification of the integers}, Trans. Amer. Math. Soc. \textbf{355}
  (2003), no.~10, 4223--4240.

\bibitem{Ve:OCA}
B.~Veli{\v{c}}kovi{\'c}, \emph{{OCA} and automorphisms of {${\mathcal
  P}(\omega) /\Fin$}}, Top. Appl. \textbf{49} (1992), 1--13.

\bibitem{Wo:Beyond}
W.H. Woodin, \emph{Beyond {$ \Sigma\sp 2\sb 1$} absoluteness}, Proceedings of
  the International Congress of Mathematicians, Vol. I (Beijing, 2002)
  (Beijing), Higher Ed. Press, 2002, pp.~515--524.

\bibitem{Wo:Pmax2}
\bysame, \emph{The axiom of determinacy, forcing axioms, and the nonstationary
  ideal}, revised ed., de Gruyter Series in Logic and its Applications, vol.~1,
  Walter de Gruyter GmbH \& Co. KG, Berlin, 2010.

\bibitem{Zap:Descriptive}
J.~Zapletal, \emph{Descriptive set theory and definable forcing}, Mem. Amer.
  Math. Soc. \textbf{167} (2004), no.~793, viii+141.

\bibitem{Zap:Forcing}
\bysame, \emph{Forcing idealized}, Cambridge University Press, 2008.

\end{thebibliography}
\def\germ{\frak} \def\scr{\cal} \ifx\documentclass\undefinedcs
  \def\bf{\fam\bffam\tenbf}\def\rm{\fam0\tenrm}\fi 
  \def\defaultdefine#1#2{\expandafter\ifx\csname#1\endcsname\relax
  \expandafter\def\csname#1\endcsname{#2}\fi} \defaultdefine{Bbb}{\bf}
  \defaultdefine{frak}{\bf} \defaultdefine{=}{\B} 
  \defaultdefine{mathfrak}{\frak} \defaultdefine{mathbb}{\bf}
  \defaultdefine{mathcal}{\cal}
  \defaultdefine{beth}{BETH}\defaultdefine{cal}{\bf} \def\bbfI{{\Bbb I}}
  \def\mbox{\hbox} \def\text{\hbox} \def\om{\omega} \def\Cal#1{{\bf #1}}
  \def\pcf{pcf} \defaultdefine{cf}{cf} \defaultdefine{reals}{{\Bbb R}}
  \defaultdefine{real}{{\Bbb R}} \def\restriction{{|}} \def\club{CLUB}
  \def\w{\omega} \def\exist{\exists} \def\se{{\germ se}} \def\bb{{\bf b}}
  \def\equivalence{\equiv} \let\lt< \let\gt>
\providecommand{\bysame}{\leavevmode\hbox to3em{\hrulefill}\thinspace}
\providecommand{\MR}{\relax\ifhmode\unskip\space\fi MR }
\providecommand{\MRhref}[2]{%
  \href{http://www.ams.org/mathscinet-getitem?mr=#1}{#2}
}
\providecommand{\href}[2]{#2}

\end{document}